\documentclass[a4j,12pt]{article}  
\usepackage{geometry}
\usepackage[T1]{fontenc} 
\usepackage{times} 
\usepackage{amsmath,amssymb,bm, mathpazo}
\usepackage{amsfonts}
\usepackage{pifont}
\usepackage{mathrsfs}
\usepackage{graphicx}
\usepackage{amscd}

\makeatletter

\newtheorem{thm}{\theoremname}[section]

\newtheorem{lem}[thm]{\lemmaname}

\newtheorem{exa}{\examplename}[section]
\newtheorem{rem}{\remarkname}[section]

\newcommand{\theoremname}{Theorem}
\newcommand{\definitionname}{Definition}
\newcommand{\lemmaname}{Lemma}
\newcommand{\corollaryname}{Corollary}
\newcommand{\axiomname}{Axiom}
\newcommand{\propositionname}{Proposition}
\newcommand{\problemname}{Question}
\newcommand{\examplename}{Example}
\newcommand{\remarkname}{Remark}

\def\tedsymbol{\vcenter{\hbox{\vrule\@height.5em\@width.5em}}}
\def\ted{{\unskip\nobreak\hfil\penalty50
 \quad\hbox{}\nobreak\hfil \hbox{$\tedsymbol$}
 \parfillskip\z@ \finalhyphendemerits\z@\par}}

\makeatother

\title{Notes on the divisibility of the class numbers of the imaginary quadratic fields $\mathbb{Q}(\sqrt{3^{2e} - 4k^n})$}
\author{
\LARGE{Akiko Ito}\footnote{The author is supported by JSPS Core-to-Core 18005.}}
\begin{document}
\date{\empty}
\maketitle
{\small {\bf Abstract.}
Let $n$ be a positive integer, $k$ an integer greater than $1$, and $x$ a positive integer with $\gcd(k, x) = 1$.
We study the divisibility of the class numbers of the imaginary quadratic fields $\mathbb{Q}(\sqrt{x^2 - 4k^n})$.
First, we treat the case where $x = 1$.
This case is studied by Le~Maohua~\cite{Le} and S.~R.~Louboutin~\cite{Lo}.
Combining their methods and results on positive integer solutions of certain Diophantine equations,
we refine their results.
Next, we treat the case where $k$ is a prime number and $x$ is a power of three.
Let $q$ be a prime number except for $3$ and let $n$ and $e$ be positive integers.
Using Y.~Kishi's method in \cite{Ki}, 
we show that the class numbers of the imaginary quadratic fields $\mathbb{Q}(\sqrt{3^{2e} - 4q^n})$ are divisible by $n$
if $n \not \equiv 2 \bmod 4$ or $n \equiv 2 \bmod 4$ with some conditions.}
\section{Introduction} \label{sec1}
For a given positive integer $n$, there exist infinitely many imaginary quadratic fields whose class numbers are divisible by $n$.
Such results were obtained by T. Nagell~\cite{Na}, N.~C.~Ankeny and S.~Chowla~\cite{AC}, R.~A.~Mollin~\cite{Mo}, etc.
Their proofs were given by constructing such quadratic fields explicitly.
We study the divisibility of the class numbers of the imaginary quadratic fields $\mathbb{Q}(\sqrt{x^2 - 4k^n})$.
First, we treat the case where $x = 1$.
We begin with stating the following result.
\begin{thm}[Gross and Rohrlich, {[5, Theorem 5.3]} and Louboutin, {[11, Theorem 1]}] \label{thm2}
Let $n$ be a positive odd integer and $k$ a positive integer greater than $1$.
Then, the class numbers of the imaginary quadratic fields $\mathbb{Q}(\sqrt{1 - 4k^n})$ are divisible by $n$.
\end{thm}
The following generalization was obtained by Le Maohua.
We denote by $h(d)$ the class number of a quadratic field $\mathbb{Q}(\sqrt{d})$.
\begin{thm}[Le Maohua, {[9, Theorem]}] \label{thm3}
Let $n$ be a positive integer and $k$ a positive integer greater than $1$.
Assume $1 - 4k^n = a^2d$, where $a$ is a positive integer and $d$ is a square-free integer less than $-3$.
Here, we suppose that $(n, k, a, d) \neq (4, 2, 3, -7)$.\\
$(1)$ \ If $n$ is even and there exist positive integers $a_1$ and $a_2$ such that $a = a_1a_2$ and 
$a_1^2 + a_2^2d = 2$ or $-2$, then $n / 2 \mid h(d)$.\\
$(2)$ \ Otherwise, $n \mid h(d)$.
\end{thm}
Recently, S.~R.~Louboutin proved the following theorem.
\begin{thm}[Louboutin, {[11, Theorem 1]}] \label{thm4}
Let $n$ be an even integer.
If at least one of the prime divisors of an odd integer $k \ge 3$ is congruent to $3$ modulo $4$, 
then the class numbers of the imaginary quadratic fields $\mathbb{Q}(\sqrt{1 - 4k^n})$ are divisible by $n$.
\end{thm}
We refine Theorems \ref{thm3} and \ref{thm4}.
We treat the case where $k$ is an odd integer including the case where all prime divisors of $k$ are congruent to $1$ modulo $4$.
We show that possible integers $n$ satisfying the assumption of Theorem \ref{thm3} (1) are only $2$, $4$, and $8$.
The details of the result are as follows.
\begin{thm} \label{thm5}
Let $n$ be an integer greater than $1$ and $k$ an odd integer greater than $1$.\\
$(1)$ \ Assume $k \neq 5$, $13$.
Then, for a given $k$, the class numbers of $\mathbb{Q}(\sqrt{1 - 4k^n})$ are divisible by $n$ except for at most one $n$.
The exceptional case is $n = 2$ or $n = 4$ and then, the class number of the field is divisible by $n/2$.\\
$(2)$ \ Assume $k = 5$.
Then, the class numbers of $\mathbb{Q}(\sqrt{1 - 4k^n})$ are divisible by $n$ except for the two cases $n = 2$ and $n = 4$.
The class numbers of the fields $\mathbb{Q}(\sqrt{1 - 4\cdot 5^2}) = \mathbb{Q}(\sqrt{-11})$ and $\mathbb{Q}(\sqrt{1 - 4\cdot 5^4}) = \mathbb{Q}(\sqrt{-51})$
are $1$ and $2$ respectively. 
These class numbers are divisible by $n/2$ but are not divisible by $n$.\\
$(3)$ \ Assume $k = 13$.
Then, the class numbers of $\mathbb{Q}(\sqrt{1 - 4k^n})$ are divisible by $n$ except for the two cases $n = 2$ and $n = 8$. 
The class numbers of the fields $\mathbb{Q}(\sqrt{1 - 4\cdot 13^2}) = \mathbb{Q}(\sqrt{-3})$ and $\mathbb{Q}(\sqrt{1 - 4\cdot 13^8}) = \mathbb{Q}(\sqrt{-6347})$ are $1$ and $28$ respectively.
These class numbers are divisible by $n/2$ but are not divisible by $n$.
\end{thm}
In Theorem \ref{thm5} (1),
there exists an imaginary quadratic field $\mathbb{Q}(\sqrt{1 - 4k^n})$ whose class number is divisible by $n/2$
but is not divisible by $n$.
We give such example.
\begin{exa} 
For $(k, n) = (29, 4)$, we have 
$$\mathbb{Q}(\sqrt{1 - 4k^n}) = \mathbb{Q}(\sqrt{-187}).$$
Since $h(-187)$ is equal to $2$, the class number $h(-187)$ is divisible by $n/2$
but is not divisible by $n$.
\end{exa}
Next, we treat the case where $k$ is a prime number and $x$ is a power of three.
Using the methods of Y. Kishi \cite{Ki} and A. Ito \cite{It}, we show the following theorem.
\begin{thm} \label{thm6}
Let $q$ be a prime number with $q \neq 3$ and let $n$ and $e$ be positive integers with $3^{2e} < 4q^n$.\\
$(1)$ \ Assume $q \equiv 1 \bmod 3$ or $n \not \equiv 2\bmod 4$.
Then, the class number of the imaginary quadratic field $\mathbb{Q}(\sqrt{3^{2e} - 4q^n})$ is divisible by $n$.\\  
$(2)$ \ Assume $q \equiv 2\bmod 3$ with $q \neq 2$ and $n \equiv 2\bmod 4$.\\
\hspace{10pt}$(2.1)$ \ If $2q^{n/2} - (-3)^e$ is not a square number, 
then the class number of the imaginary quadratic field $\mathbb{Q}(\sqrt{3^{2e} - 4q^n})$ is divisible by $n$.\\
\hspace{10pt}$(2.2)$ \ If $2q^{n/2} - (-3)^e$ is a square number, 
then the class number of the imaginary quadratic field $\mathbb{Q}(\sqrt{3^{2e} - 4q^n})$ is divisible by $n / 2$.\\
$(3)$ \ Assume $q = 2$ and $n \equiv 2\bmod 4$.\\
\hspace{10pt}$(3.1)$ \ When $(n, e) \neq (6, 2)$, we have the following:\\
\hspace{20pt}$(3.1.1)$ \ If $e \equiv 0 \bmod 2$, then the class number of the imaginary quadratic field $\mathbb{Q}(\sqrt{3^{2e} - 4q^n})$ is divisible by $n$.\\
\hspace{20pt}$(3.1.2)$ \ If $e \equiv 1 \bmod 2$ and $2^{(n/2)+1} - 3^e$ is not a square number,
then the class number of $\mathbb{Q}(\sqrt{3^{2e} - 4q^n})$ is divisible by $n$.\\
\hspace{20pt}$(3.1.3)$ \ If $e \equiv 1 \bmod 2$ and $2^{(n/2)+1} - 3^e$ is a square number,
then the class number of $\mathbb{Q}(\sqrt{3^{2e} - 4q^n})$ is divisible by $n / 2$.\\
\hspace{10pt}$(3.2)$ \ When $(n, e) = (6, 2)$, we have $\mathbb{Q}(\sqrt{3^{2e} - 4q^n}) = \mathbb{Q}(\sqrt{-7})$ and
$h(-7) = 1$ not divisible by $6$.
\end{thm}
In Theorem \ref{thm6} (2.2) and (3.1.3), 
there exist imaginary quadratic fields whose class numbers are divisible by $n/2$
but are not divisible by $n$.
We give such examples.
\begin{exa} \label{exa7}
$(1)$ \ For $(q, n, e) = (5, 2, 2)$, we have 
$$\mathbb{Q}(\sqrt{3^{2e} - 4q^n}) = \mathbb{Q}(\sqrt{-19})$$
and $2q^{n/2} - (-3)^e = 1$, a square number.
Since $h(-19)$ is equal to $1$, the class number $h(-19)$ is divisible by $n/2$
but is not divisible by $n$.\\
$(2)$ \ For $(q, n, e) = (2, 2, 1)$, we have 
$$\mathbb{Q}(\sqrt{3^{2e} - 4q^n}) = \mathbb{Q}(\sqrt{-7})$$
and $2^{(n/2)+1} - 3^e = 1$, a square number.
Since $h(-7)$ is equal to $1$, the class number $h(-7)$ is divisible by $n/2$
but is not divisible by $n$.
\end{exa}
This paper is organized as follows.
In Section \ref{sec2}, we give a proof of Theorem \ref{thm5}. 
In Section \ref{sec3}, we show Theorem \ref{thm6}.
In Section \ref{sec4}, we prove other results on the divisibility of the class numbers of the imaginary quadratic fields $\mathbb{Q}(\sqrt{x^2 - 4k^n})$.
\section{Proof of Theorem \ref{thm5}} \label{sec2}
In this section, we show Theorem \ref{thm5}.
Let $d$ be the square-free part of $1 - 4k^n$.
Note that $d$ is a negative odd integer.
We can write $1 - 4k^n = a^2d$ for some positive odd integer $a$.
Put
$$\tau := \frac{1 + a\sqrt{d}}{2} \in \mathcal{O}_{\mathbb{Q}(\sqrt{d})}.$$
The following lemma is essential for the proof of Theorem \ref{thm5}.
\begin{lem} \label{lem2.1}
Let $n$ be an integer greater than $1$ and $k$ an odd integer greater than $1$.
Then, we have the following:\\
$(1) \ \pm \tau$ is not a $p$th power in $\mathcal{O}_{\mathbb{Q}(\sqrt{d})}$ for any odd prime number $p$ dividing $n$.\\
$(2)$ \ When $n$ is even, we have the following:\\
\hspace{10pt}$(2.A)$ \ Assume $k \neq 13$.\\
\hspace{20pt}$(2.A.1) \ \pm \tau$ is not a square number in $\mathcal{O}_{\mathbb{Q}(\sqrt{d})}$ for $n \neq 2$, $4$.\\
\hspace{20pt}$(2.A.2)$ \ If $k \neq 5$, then $\pm \tau$ is not a square number in $\mathcal{O}_{\mathbb{Q}(\sqrt{d})}$
 for at least one of the cases where $n = 2$ or $n = 4$.\\
\hspace{10pt}$(2.B)$ \ Assume $k = 13$.
If $n \neq 2$, $8$, then $\pm \tau$ is not a square number in $\mathcal{O}_{\mathbb{Q}(\sqrt{d})}$.    
\end{lem}
To prove this, we prepare some lemmas in Section \ref{sec2.1}.
Using these, we show Lemma \ref{lem2.1} in Section \ref{sec2.2}.
Theorem \ref{thm5} follows from Lemma \ref{lem2.1}.
We state this in Section \ref{sec2.3}.
\subsection{Preliminaries} \label{sec2.1}
In this section, we prepare some lemmas.
We say that $\tau_1 \in \mathcal{O}_{\mathbb{Q}(\sqrt{d})}$ is associated with $\tau_2 \in \mathcal{O}_{\mathbb{Q}(\sqrt{d})}$
if there exists $\theta \in \mathcal{O}_{\mathbb{Q}(\sqrt{d})}^{\times}$ such that $\tau_2 = \theta \tau_1$ (cf. \cite{Lo}).
Louboutin proved the following lemmas.
\begin{lem}[Louboutin, {[11, Proposition 3]}] \label{lem2.2}
Assume that $\tau_3 \in \mathcal{O}_{\mathbb{Q}(\sqrt{d})}$ and $Tr(\tau_3) = 1$.
If $\tau_3$ is associated with a $p$th power for some odd prime number $p$, then $\tau_3$ is a unit of $\mathcal{O}_{\mathbb{Q}(\sqrt{d})}$, 
where $Tr(\tau_3)$ denotes the trace of $\tau_3$.
\end{lem}
\begin{lem}[Louboutin, {[11, Lemma 4]}] \label{lem2.3}
If $\tau_4 \in \mathcal{O}_{\mathbb{Q}(\sqrt{d})}$, then $\tau_4$ is a square number in $\mathcal{O}_{\mathbb{Q}(\sqrt{d})}$
if and only if there exists $c \in \mathbb{Z}$ such that $N(\tau_4) = c^2$ and $Tr(\tau_4) + 2c$ is a square number in $\mathbb{Z}$,
where $N(\tau_4)$ denotes the norm of $\tau_4$.
\end{lem}
These are essential for the proof of Lemma \ref{lem2.1}.
To prove Lemma \ref{lem2.1} (2), we use some results on positive integer solutions $(x, y, z)$ of the equation $x^2 + 1 = 2y^z$.
We state this here.
\begin{lem} \label{lem2.4}
$(1)$ \ The equation
$$x^2 + 1 = 2y^z$$
has no positive integer solution $(x, y, z)$ such that $y$ is greater than $1$ and $z$ is an odd integer greater than $1$ (see \cite[p.~90]{Ri} or \cite[p.~65]{BS}).\\
$(2)$ \ The equation
$$x^2 + 1 = 2y^4$$
has no positive integer solution $(x, y)$ with $y > 1$ if $y \neq 13$ (see \cite[p.~141]{Ri} or \cite[p.~65]{BS}).\\
$(3)$ \ The equation
$$x^2 + 1 = 2\cdot 13^z$$
has two positive integer solutions $(x, z) = (5, 1)$ and $(x, z) = (239, 4)$ (see \cite[p.~141]{Ri} or \cite[p.~65]{BS}).\\
$(4)$ \ Let $k$ be a positive integer. If the equation
$$x^2 + 1 = 2k^z$$
has two positive integer solutions $(x, z)$ with $z = 1, 2$, then $k = 1, 5$ (see \cite[p.~65]{BS}).
\end{lem}
\subsection{Proof of Lemma \ref{lem2.1}} \label{sec2.2}
In this section, we show Lemma \ref{lem2.1} by using Lemmas \ref{lem2.2}, \ref{lem2.3} and \ref{lem2.4}.\\
~\\
{\bf Proof of Lemma \ref{lem2.1}.} \ \ Note that $Tr(\tau) = 1$ and $N(\tau) = k^n \neq 1$.
Lemma \ref{lem2.1} (1) follows from Lemma \ref{lem2.2} immediately.
Assume that $\tau$ is a square number in $\mathcal{O}_{\mathbb{Q}(\sqrt{d})}$.
It follows from Lemma \ref{lem2.3} that $2k^{n/2} + 1$ is a square number in $\mathbb{Z}$.
Since $k$ is odd, $2k^{n/2} + 1 \equiv 3 \bmod 4$, a contradiction.
Next assume $- \tau$ is a square number in $\mathcal{O}_{\mathbb{Q}(\sqrt{d})}$.
It follows from Lemma \ref{lem2.3} that $2k^{n/2} - 1$ is a square number in $\mathbb{Z}$.
Then, 
$$u^2 + 1 = 2k^{n/2}$$
for some positive integer $u$.\\
(i) \ We treat the case where $n \equiv 2 \bmod 4$.
We write $n = 4s_1 + 2$ for some non-negative integer $s_1$.
Then,
$$u^2 + 1 = 2k^{2s_1 + 1}.$$
We see from Lemma \ref{lem2.4} (1) that this is impossible when $s_1 \neq 0$, that is, $n \neq 2$.\\
(ii) \ We treat the case where $n \equiv 4 \bmod 8$.
We write $n = 8s_2 + 4$ for some non-negative integer $s_2$.
Then,
$$u^2 + 1 = 2(k^2)^{2s_2 + 1}.$$
We see from Lemma \ref{lem2.4} (1) that this is impossible when $s_2 \neq 0$, that is, $n \neq 4$.\\
(iii) \ We treat the case where $n \equiv 0 \bmod 8$.
We write $n = 8s_3$ for some positive integer $s_3$.
Then,
$$u^2 + 1 = 2(k^{s_3})^4.$$
We see from Lemma \ref{lem2.4} (2) that this is impossible when $k \neq 13$.

Lemma \ref{lem2.1} (2.A.1) follows from the above discussions.
We can show Lemma \ref{lem2.1} (2.A.2) easily, 
as it follows from Lemma \ref{lem2.4} (4) that the equation
$$u^2 + 1 = 2k^{n/2}$$
has two positive integer solutions $(u, n)$ with $n = 2$, $4$ only if $k = 5$.
Here, we treat the case where $k = 13$.
We see from Lemma \ref{lem2.4} (3) that the equation
$$u^2 + 1 = 2\cdot 13^{n/2}$$
has two positive integer solutions $(u, n) = (5, 2)$, $(239, 8)$.
Lemma \ref{lem2.1} (2.B) follows from this.
The proof of Lemma \ref{lem2.1} is completed.
\subsection{Proof of Theorem \ref{thm5}} \label{sec2.3}
In this section, we show Theorem \ref{thm5} by using Lemma \ref{lem2.1}.
We need the following lemma to give the proof.
\begin{lem} \label{lem2.5}
$(1)$ \ The equation
$$x^4 - 2y^2 = 1$$
has no positive integer solution $(x, y)$.\\
$(2)$ \ The equation
$$x^4 - 2y^2 = -1$$
has only one positive integer solution $(x, y) = (1, 1)$.
\end{lem}
\begin{proof}
If the equation $X^4 - Y^4 = 2Z^2$ has integer solutions $(X, Y, Z)$, then $XYZ = 0$ (see \cite[(A14.5)]{Ri}).
The statement (1) follows from this.
We will show (2).
We will prove that 
if the equation $X^4 + Y^4 = 2Z^2$ has positive integer solutions $(X, Y, Z)$ with $\gcd(X, Z) = 1$, then $X = Z = 1$.
The statement (2) follows from this.
Assume that $(X, Y, Z)$ is a positive integer solution with $\gcd(X, Z) = 1$ of the equation $X^4 + Y^4 = 2Z^2$.
Then, 
$$(X^2 + Y^2)^2 = 2(Z^2 + X^2Y^2)$$
and
$$(X^2 - Y^2)^2 = 2(Z^2 - X^2Y^2).$$
Multiplying these,
$$\biggl(\frac{X^4 - Y^4}{2}\biggr)^2 = Z^4 - X^4Y^4.$$
Fermat proved that the equation $X_1^4 - Y_1^4 = Z_1^2$ has no solution in non-zero integers (cf. \cite[(P3.2)]{Ri}).
It follows from $XYZ \neq 0$ that $X^4 = Y^4$.
Then, $Z = X^2$.
We see from $\gcd(X, Z) = 1$ that $X = Z = 1$.
The proof of Lemma \ref{lem2.5} is completed.
\end{proof}
\begin{rem}
We give a remark on the proof of Lemma \ref{lem2.5} (2).
P.~Ribenboim~\cite[(A14.4)]{Ri} states that the equation $X^4 + Y^4 = 2Z^2$ has only one integer solution $(X, Y, Z) = (0, 0, 0)$.
However, $(X, Y, Z) = (1, 1, 1)$ is also an integer solution of this equation. 
We correct his proof here.
\end{rem}
{\bf Proof of Theorem \ref{thm5}.} \ \ Let $r$ be any odd prime number dividing $k$.
Since $\gcd (r, d) = 1$ and $r \nmid \tau$ hold,
the prime number $r$ splits in $\mathbb{Q}(\sqrt{d})$.
Note that ideals $(\tau)$ and $(\overline{\tau})$ are coprime and $N(\tau) = k^n$,
where $\overline{\tau}$ is the conjugate of $\tau$.
Then,
$$(\tau) = \mathfrak{I}^n$$
for some ideal $\mathfrak{I}$ of $\mathcal{O}_{\mathbb{Q}(\sqrt{d})}$.
Let $s$ be the order of the ideal class containing $\mathfrak{I}$.
We can write $n = sn'$ for some integer $n'$.
Since $\mathfrak{I}^s$ is principal, there exists some element $\rho \in \mathcal{O}_{\mathbb{Q}(\sqrt{d})}$
such that $\mathfrak{I}^s = (\rho)$.
Then,
\begin{equation} \label{eq2.0}
(\tau) = (\mathfrak{I}^s)^{n'} = (\rho)^{n'} = (\rho^{n'}).
\end{equation}                       
We see from $d \equiv 1 \bmod 4$ that $\mathbb{Q}(\sqrt{d}) \neq \mathbb{Q}(\sqrt{-1})$.
Here, we need the following lemma.
\begin{lem} \label{lem2.6}
$(1)$ \ Assume $k = 5$ and $n \neq 2$, $4$.
Then, $\mathbb{Q}(\sqrt{d}) \neq \mathbb{Q}(\sqrt{-3})$.\\
$(2)$ \ Assume $k = 13$ and $n \neq 2$, $8$.
Then, $\mathbb{Q}(\sqrt{d}) \neq \mathbb{Q}(\sqrt{-3})$.\\
$(3)$ \ Assume $k \neq 5$, $13$ and $n \neq 2$.
Then, $\mathbb{Q}(\sqrt{d}) \neq \mathbb{Q}(\sqrt{-3})$.
\end{lem}
We show this lemma later.\\
(i) \ We treat the case where $k = 5$.\\
(i-1) \ Suppose $n \neq 2$, $4$.
By Lemma \ref{lem2.6} (1), 
$$\mathbb{Q}(\sqrt{d}) \neq \mathbb{Q}(\sqrt{-1}), \mathbb{Q}(\sqrt{-3}).$$
Then,
$$\tau = \pm \rho^{n'}.$$
It follows from Lemma \ref{lem2.1} (1) and (2.A.1) that $n' = 1$.
Then, $s = n$.\\
(i-2) \ Suppose $n = 2$.
Then, we have $\mathbb{Q}(\sqrt{1 - 4\cdot 5^2}) = \mathbb{Q}(\sqrt{-11})$.
The class number $h(-11)$ is $1$ and $n/2$ divides $h(-11)$.\\
(i-3) \ Suppose $n = 4$.
Then, we have $\mathbb{Q}(\sqrt{1 - 4\cdot 5^4}) = \mathbb{Q}(\sqrt{-51})$.
The class number $h(-51)$ is $2$ and $n/2$ divides $h(-51)$.\\
(ii) \ We treat the case where $k = 13$.\\
(ii-1) \ Suppose $n \neq 2$, $8$.
By Lemma \ref{lem2.6} (2), 
$$\mathbb{Q}(\sqrt{d}) \neq \mathbb{Q}(\sqrt{-1}), \mathbb{Q}(\sqrt{-3}).$$
Then,
$$\tau = \pm \rho^{n'}.$$
It follows from Lemma \ref{lem2.1} (1) and (2.B) that $n' = 1$.
Then, $s = n$.\\
(ii-2) \ Suppose $n = 2$.
Then, we have $\mathbb{Q}(\sqrt{1 - 4\cdot 13^2}) = \mathbb{Q}(\sqrt{-3})$.
The class number $h(-3)$ is $1$ and $n/2$ divides $h(-3)$.\\
(ii-3) \ Suppose $n = 8$.
Then, we have $\mathbb{Q}(\sqrt{1 - 4\cdot 13^8}) = \mathbb{Q}(\sqrt{-6347})$.
The class number $h(-6347)$ is $28$ and $n/2$ divides $h(-6347)$.\\
(iii) \ We treat the case where $k \neq 5$, $13$.\\
(iii-1) \ Suppose $n \neq 2$, $4$.
By Lemma \ref{lem2.6} (3), 
$$\mathbb{Q}(\sqrt{d}) \neq \mathbb{Q}(\sqrt{-1}), \mathbb{Q}(\sqrt{-3}).$$
Then,
$$\tau = \pm \rho^{n'}.$$
It follows from Lemma \ref{lem2.1}  (1) and (2.A.1) that $n' = 1$.
Then, $s = n$.\\
(iii-2) \ Suppose $n = 2$ or $4$.
We need the following lemma.
\begin{lem} \label{lem2.7}
If $n = 4$, then $n' \neq 4$.
\end{lem}
\begin{proof}
The method of the proof is based on the one in \cite[p. 726]{Le}.
By Lemma \ref{lem2.6} (3), if $n = 4$,  
$$\mathbb{Q}(\sqrt{d}) \neq \mathbb{Q}(\sqrt{-1}), \mathbb{Q}(\sqrt{-3}).$$
Then,
$$\tau = \pm \rho^{n'}.$$
Suppose $n' = 4$.
Then,
$$\pm \frac{1 + a\sqrt{d}}{2} = \biggl(\frac{u + v\sqrt{d}}{2}\biggr)^4$$
for some odd integers $u$ and $v$.
Note that
$$\pm \overline{\tau} = \pm \frac{1 - a\sqrt{d}}{2} = \biggl(\frac{u - v\sqrt{d}}{2}\biggr)^4.$$
We see
\begin{equation} \label{eq2.1}
 \begin{split}
\pm 1 = \pm (\tau + \overline{\tau}) &= \biggl(\frac{u + v\sqrt{d}}{2}\biggr)^4 + \biggl(\frac{u - v\sqrt{d}}{2}\biggr)^4 \\
                                                      &= \frac{1}{8}(u^4 + 6u^2v^2d + v^4d^2). 
 \end{split}
\end{equation}                       
Then,
$$u^4 - 2\biggl(\frac{v^2d + 3u^2}{4}\biggr)^2 = -\frac{1}{8}(u^4 + 6u^2v^2d + v^4d^2) = \mp 1.$$
Since $u$ and $v$ are odd, $u^2 \equiv v^2 \equiv 1 \bmod 4$.
Moreover, $d \equiv 1 \bmod 4$.
Then, $v^2d + 3u^2 \equiv 0 \bmod 4$,
that is, $\dfrac{v^2d + 3u^2}{4}$ is an integer.
Suppose $3u^2 = -v^2d$.
It follows from equation (\ref{eq2.1}) that $\gcd(u, v) = 1$.
Since $d$ is a square-free integer and $d \neq -3$,
this is impossible.
Then, 
$$v^2d + 3u^2 \neq 0.$$
By Lemma \ref{lem2.5}, we have $u^2 = 1$ and
$$\frac{v^2d + 3u^2}{4} = \pm 1.$$
Then, $v^2d = 1$ or $-7$.
Since $v^2d$ is negative, $(v, d) = (\pm 1, -7)$.
Note that
\begin{equation} \label{eq2.2}
 \begin{split}
\pm a = \pm \frac{1}{\sqrt{d}}(\tau - \overline{\tau}) &= \frac{1}{\sqrt{d}}\biggl\{\biggl(\frac{u + v\sqrt{d}}{2}\biggr)^4 
- \biggl(\frac{u - v\sqrt{d}}{2}\biggr)^4\biggr\} \notag \\
                                                      &= \frac{1}{2}uv(u^2 + v^2d). \notag
 \end{split}
\end{equation}                       
Then, 
$$a = \pm \frac{1}{2}uv(1 - 7) = \mp 3uv = -3 \ {\rm or} \ 3.$$
Since $a$ is positive, $a = 3$.
Then, 
$$1 - 4k^4 = a^2d = -63,$$
that is, $k = 2$, a contradiction.
Hence, $n' \neq 4$.
The proof of Lemma \ref{lem2.7} is completed.
\end{proof}
(iii-2-1) \ Assume $\mathbb{Q}(\sqrt{1 - 4k^2}) \neq \mathbb{Q}(\sqrt{-3})$.
By Lemma \ref{lem2.6} (3), we have
$$\mathbb{Q}(\sqrt{1 - 4k^n}) \neq \mathbb{Q}(\sqrt{-1}), \mathbb{Q}(\sqrt{-3}).$$
Then,
$$\tau = \pm \rho^{n'}.$$
We see from Lemma \ref{lem2.1} (2.A.2) that $n' = 1$ for at least one of the cases where $n = 2$ or $n = 4$.
For such $n$, we have $s = n$.
If $n = 2$ and $n' \neq 1$, then $n' = 2$, that is, $s = n/2$.
It follows from Lemma \ref{lem2.7} that $n' = 2$ if $n = 4$ and $n' \neq 1$.
Then, $s = n/2$ if $n = 4$ and $n' \neq 1$.\\
(iii-2-2) \ Assume $\mathbb{Q}(\sqrt{1 - 4k^2}) = \mathbb{Q}(\sqrt{-3})$.
When $n = 2$, we write $n = sn_1'$ for some integer $n'_1$.
We see from equation (\ref{eq2.0}) that 
$$\tau = \pm \rho^{n_1'}, \ \pm \omega \rho^{n_1'}, \ \pm \omega^2\rho^{n_1'},$$
where $\omega = \dfrac{-1 + \sqrt{-3}}{2}$.
Note that $h(-3) = 1$.
Then, $s = 1$ and $n'_1 = 2$.
Since $\omega = \omega^4$ and $n'_1 = 2$ hold,
$\pm \tau$ is a square number in $\mathcal{O}_{\mathbb{Q}(\sqrt{-3})}$.
Here, we treat the case where $n = 4$.
We write $n = sn_2'$ for some integer $n'_2$.
By equation (\ref{eq2.0}) and Lemma \ref{lem2.6} (3), we have 
$$\tau = \pm \rho^{n'_2}.$$
It follows from Lemma \ref{lem2.1} (2.A.2) and the above discussion that $n'_2 = 1$ and $s = n$.
Therefore, $s = n/2$ for $n = 2$ and $s = n$ for $n = 4$.
The proof of Theorem \ref{thm5} is completed.\\
~\\
Finally, we will show Lemma \ref{lem2.6}, whose proof was postponed.\\
~\\
{\bf Proof of Lemma \ref{lem2.6}.} \ \ Assume $\mathbb{Q}(\sqrt{d}) = \mathbb{Q}(\sqrt{-3})$.
Note that $h(-3) = 1$.
Then, $s = 1$ and $n' = n$.
By equation (\ref{eq2.0}), we have
$$\tau = \pm \rho^{n'}, \ \pm \omega \rho^{n'}, \ \pm \omega^2\rho^{n'}.$$
It follows from Lemma \ref{lem2.2} and $\tau \not \in \mathcal{O}_{\mathbb{Q}(\sqrt{d})}^{\times}$
that $n'$ must be $1$ or a power of two.
Since $\omega = \omega^4$ holds, we can write
$$\tau = \pm \rho^{n'}, \ \pm (\omega^2)^2 \rho^{n'}, \ \pm \omega^2\rho^{n'}.$$

First, we show (1).
It follows from Lemma \ref{lem2.1} (2.A.1) that $n' = 1$ if $n \neq 2$, $4$, a contradiction.

Secondly, we show (2).
It follows from Lemma \ref{lem2.1} (2.B) that $n' = 1$ if $n \neq 2$, $8$, a contradiction.

Finally, we show (3).
It follows from Lemma \ref{lem2.1} (2.A.1) that $n' = 1$ if $n \neq 2$, $4$, a contradiction.
Now, we treat the case where $n = 4$.
In this case, $n' = 4$.
Since $\omega = \omega^4$ holds,
$\rho^4$, $\omega \rho^4$, and $\omega^2 \rho^4$ are fourth power in $\mathcal{O}_{\mathbb{Q}(\sqrt{-3})}$.
Then,
$$\pm \frac{1 + a\sqrt{-3}}{2} = \biggl(\frac{u' + v' \sqrt{-3}}{2}\biggr)^4$$
for some odd integers $u'$ and $v'$.
We see
\begin{equation} \label{eq2.3}
 \begin{split}
\pm 1 = \pm (\tau + \overline{\tau}) = \frac{1}{8}(u'^4 - 18u'^2v'^2 + 9v'^4).
 \end{split}
\end{equation}                       
Then, 
$$u'^4 - 2\biggl(\frac{-3v'^2 + 3u'^2}{4}\biggr)^2 = \mp 1.$$
Since $u'^2 \equiv v'^2 \equiv 1 \bmod 4$ holds, $\dfrac{-3v'^2 + 3u'^2}{4}$ is an integer.
Suppose $-3v'^2 + 3u'^2 = 0$.
Then, $u'^2 = v'^2$.
It follows from equation (\ref{eq2.3}) that $\gcd(u', v') = 1$.
Then, $u'^2 = v'^2 = 1$.
We see from this that
$$\frac{u' + v' \sqrt{-3}}{2} \in \mathcal{O}_{\mathbb{Q}(\sqrt{-3})}^{\times}.$$
Since $\pm \tau$ is not contained in $\mathcal{O}_{\mathbb{Q}(\sqrt{-3})}^{\times}$,
this is impossible.
Then, 
$$-3v'^2 + 3u'^2 \neq 0.$$
By Lemma \ref{lem2.5}, we find $u'^2 = 1$ and
$$\frac{-3v'^2 + 3u'^2}{4} = \pm 1.$$
Since $\dfrac{-3v'^2 + 3u'^2}{4} \in 3\mathbb{Z}$ and $\pm 1 \not \in 3\mathbb{Z}$ hold,
this is a contradiction.
The proof of Lemma \ref{lem2.6} is completed.
\section{Proof of Theorem \ref{thm6}} \label{sec3}
In this section, we show Theorem \ref{thm6}.
We sketch the outline of the proof of Thorem \ref{thm6} here.

First, we treat the cases (1), (2.1), (3.1.1), and (3.1.2) of Theorem \ref{thm6}.
Let
$$\alpha := \frac{3^e + \sqrt{3^{2e} - 4q^n}}{2} \in \mathcal{O}_{\mathbb{Q}(\sqrt{3^{2e} - 4q^n})}.$$
Since $\gcd(q, 3^{2e} - 4q^n) = 1$ and $q \nmid \alpha$ hold, the prime number $q$ splits in $\mathbb{Q}(\sqrt{3^{2e} - 4q^n})$.
Note that ideals $(\alpha)$ and $(\overline{\alpha})$ are coprime and $N(\alpha) = q^n$,
where $\overline{\alpha}$ is the conjugate of $\alpha$.
Then, 
$$(\alpha) = \wp^n,$$
where $\wp$ is the prime ideal of $\mathcal{O}_{\mathbb{Q}(\sqrt{3^{2e} - 4q^n})}$ over $q$.
We will prove the order of the ideal class containing $\wp$ is $n$.
It is essential for this proof to show that $\pm \alpha$ is not a $p$th power in $\mathcal{O}_{\mathbb{Q}(\sqrt{3^{2e} - 4q^n})}$
for any prime number $p$ dividing $n$ (Lemma \ref{lem3.3}).
In fact, we can prove this by checking that there is no integer solution $(u, v)$ of the equation
$$\pm \alpha = \biggl(\frac{u + v\sqrt{d}}{2}\biggr)^p,$$
where $d$ is the square-free part of $3^{2e} - 4q^n$.

Next, we treat the cases (2.2) and (3.1.3) of Theorem \ref{thm6}.
We obtain $(\alpha) = \wp^n$ in a way similar to the above.
We will make sure the order of the ideal class containing $\wp$ is $n / 2$ or $n$ by showing that $\pm \alpha$ is not a $p$th power in $\mathcal{O}_{\mathbb{Q}(\sqrt{3^{2e} - 4q^n})}$ for any odd prime number $p$ dividing $n$ (Lemma \ref{lem3.3}).

To prove Lemma \ref{lem3.3}, 
we use a result of Y.~Bugeaud and T.~N.~Shorey on positive integer solutions of certain Diophantine equations \cite{BS}.
In Section \ref{sec3.1}, we state their result.
In Section \ref{sec3.2}, we show Lemma \ref{lem3.3}.
In Section \ref{sec3.3}, we prove Theorem \ref{thm6}.
\subsection{Result of Bugeaud and Shorey} \label{sec3.1}
In this section, we state a result of Bugeaud and Shorey \cite{BS}.

We denote by $(F_n)$ the Fibonacci sequence defined by
$F_0 := 0$, $F_1 := 1$, and satisfying $F_{n+2} := F_{n+1} + F_{n}$ for all $n \ge 0$.
We denote by $(L_n)$ the Lucas sequence defined by $L_0 := 2$, $L_1 := 1$, and satisfying $L_{n+2} := L_{n+1} + L_{n}$ for all $n \ge 0$.
Then set
$$\mathcal{F} := \{(F_{h_1-2\varepsilon}, L_{h_1+\varepsilon}, F_{h_1}) \mid h_1 \in \mathbb{N} \ s.t. \ h_1 \ge 2 \ {\rm and} \ \varepsilon \in \{\pm 1\}\}$$
and
$$\mathcal{G} := \{(1, 4p_1^{h_2} - 1, p_1) \mid p_1 \ {\rm is \ a \ prime \ number \ and} \ h_2 \in \mathbb{N}\}.$$
Bugeaud and Shorey gave the following theorem. 
\begin{thm}[Bugeaud and Shorey, {[2, Theorem 1]}]\label{thm3.1} 
Let $D_1$ and $D_2$ be coprime positive integers and let $p$ be a prime number with $\gcd(D_1D_2, p) = 1$.
Let $\gamma \in \{1, \sqrt{2}, 2\}$ be such that $\gamma = 2$ if $p = 2$.
We assume that $D_2$ is odd if $\gamma \in \{\sqrt{2}, 2\}$.
Then, the number of positive integer solutions $(x, y)$ of the equation
$$D_1x^2 + D_2 = \gamma^2p^y$$
is at most one except for
\begin{eqnarray}
(\gamma, D_1, D_2, p) \in \mathcal{E} := \left\{
\begin{array}{c}
(2, 13, 3, 2), (\sqrt{2}, 7, 11, 3), (1, 2, 1, 3), (2, 7, 1, 2), \nonumber \\
(\sqrt{2}, 1, 1, 5), (\sqrt{2}, 1, 1, 13), (2, 1, 3, 7). \nonumber \\
\end{array}
\right \}
\end{eqnarray}
or
$$(D_1, D_2, p) \in \mathcal{F}\cup \mathcal{G}\cup \mathcal{H}_{\gamma},$$
where $\mathcal{H}_{\gamma}$ denotes the set  
\begin{eqnarray}
\mathcal{H}_{\gamma} := \left\{(D_1, D_2, p)\hspace{5pt}
\begin{array}{|c}
there \ exist \ positive \ integers \ s_0 \ and \ t_0 \nonumber \\
such \ that \ D_1s_0^2 + D_2 = \gamma^2p^{t_0} \ and \nonumber \\ 
3D_1s_0^2 - D_2 = \pm \gamma^2 \nonumber \\
\end{array}
\right \}.
\end{eqnarray}
\end{thm}
\subsection{Preliminaries} \label{sec3.2}
In this section, we show the key lemma (Lemma \ref{lem3.3}).
We need the following lemma to prove this.
\begin{lem} \label{lem3.2}
Let $q$ be a prime number with $q \neq 3$ and $e$ a positive integer.
Then, for any given positive odd integer $D_1$, the number of positive integer solutions $(x, y)$ of 
the equation 
$$D_1x^2 + 3^{2e} = 4q^y$$
is at most one.
\end{lem}
\begin{proof}
We will show 
$$(\gamma, D_1, D_2, p) = (2, D_1, 3^{2e}, q) \not \in \mathcal{E}$$
and
$$(D_1, D_2, p) = (D_1, 3^{2e}, q) \not \in \mathcal{F}\cup \mathcal{G}\cup \mathcal{H}_2$$
to use Theorem \ref{thm3.1} with $\gamma = 2$.
We see $(2, D_1, 3^{2e}, q) \not \in \mathcal{E}$ easily.
Suppose $(D_1, 3^{2e}, q) \in \mathcal{F}$.
There exists an integer $h_1$ greater than $1$ 
such that $F_{h_1-2\varepsilon} = D_1$, $L_{h_1+\varepsilon} = 3^{2e}$, and $F_{h_1} = q$, where $\varepsilon = \pm 1$.
J.~H.~E.~Cohn~\cite{Co1} proved that $L_1 = 1$ and $L_3 = 4$ are the only squares in the Lucas sequence.
Then, $L_{h_1+\varepsilon} = 3^{2e}$ is impossible.
Therefore, $(D_1, 3^{2e}, q) \not \in \mathcal{F}$.
Next suppose $(D_1, 3^{2e}, q) \in \mathcal{G}$.
Then, $4q^{h_2} - 1 = 3^{2e}$ for some positive integer $h_2$.
Since $4q^{h_2} - 1 \equiv 3 \bmod 4$ and $3^{2e} \equiv 1 \bmod 4$ hold, this is impossible.
Then, $(D_1, 3^{2e}, q) \not \in \mathcal{G}$.
Assume $(D_1, 3^{2e}, q) \in \mathcal{H}_2$.
Then, there exists a positive integer $s_0$ such that $3D_1s_0^2 - 3^{2e} = \pm 4$.
Since $e$ is positive, this is impossible.
Therefore, $(D_1, 3^{2e}, q) \not \in \mathcal{H}_2$.
The proof of Lemma \ref{lem3.2} is completed.
\end{proof}
Using this, we show the following lemma.
\begin{lem} \label{lem3.3}
Let $q$ be a prime number with $q \neq 3$ and let $n$ and $e$ be positive integers with $3^{2e} < 4q^n$.
Define
$$\alpha := \frac{3^e + \sqrt{3^{2e} - 4q^n}}{2} \in \mathcal{O}_{\mathbb{Q}(\sqrt{3^{2e} - 4q^n})}.$$
Then, we have the following:\\
$(1)$ \ Assume $q$ is an odd prime number.\\
\hspace{10pt}$(1.a)$ \ If $n \not \equiv 2 \bmod 4$, then $\pm \alpha$ is not a $p$th power in $\mathcal{O}_{\mathbb{Q}(\sqrt{3^{2e} - 4q^n})}$
for any prime number $p$ dividing $n$.\\
\hspace{10pt}$(1.b)$ \ If $q \equiv 1 \bmod 3$, then $\pm \alpha$ is not a $p$th power in $\mathcal{O}_{\mathbb{Q}(\sqrt{3^{2e} - 4q^n})}$
for any prime number $p$ dividing $n$.\\
\hspace{10pt}$(1.c)$ \ If $q \equiv 2 \bmod 3$, $n \equiv 2 \bmod 4$, and $2q^{n/2} - (-3)^e$ is not a square number, 
then $\pm \alpha$ is not a $p$th power in $\mathcal{O}_{\mathbb{Q}(\sqrt{3^{2e} - 4q^n})}$ for any prime number $p$ dividing $n$.\\
\hspace{10pt}$(1.d)$ \ If $q \equiv 2 \bmod 3$, $n \equiv 2 \bmod 4$, and $2q^{n/2} - (-3)^e$ is a square number, 
then $\pm \alpha$ is not a $p$th power in $\mathcal{O}_{\mathbb{Q}(\sqrt{3^{2e} - 4q^n})}$ for any odd prime number $p$ dividing $n$.\\  
$(2)$ \ Assume $q = 2$.\\
\hspace{10pt}$(2.a)$ \ If $n \not \equiv 2 \bmod 4$, then $\pm \alpha$ is not a $p$th power in $\mathcal{O}_{\mathbb{Q}(\sqrt{3^{2e} - 4q^n})}$
for any prime number $p$ dividing $n$.\\
\hspace{10pt}$(2.b)$ \ If $n \equiv 2 \bmod 4$ and $e \equiv 0 \bmod 2$ with $(n, e) \neq (6, 2)$,
then $\pm \alpha$ is not a $p$th power in $\mathcal{O}_{\mathbb{Q}(\sqrt{3^{2e} - 4q^n})}$ for any prime number $p$ dividing $n$.\\
\hspace{10pt}$(2.c)$ \ If $n \equiv 2 \bmod 4$, $e \equiv 1 \bmod 2$, and $2^{(n/2)+1} - 3^e$ is not a square number, 
then $\pm \alpha$ is not a $p$th power in $\mathcal{O}_{\mathbb{Q}(\sqrt{3^{2e} - 4q^n})}$ for any prime number $p$ dividing $n$.\\ 
\hspace{10pt}$(2.d)$ \ If $n \equiv 2 \bmod 4$, $e \equiv 1 \bmod 2$, and $2^{(n/2)+1} - 3^e$ is a square number, 
then $\pm \alpha$ is not a $p$th power in $\mathcal{O}_{\mathbb{Q}(\sqrt{3^{2e} - 4q^n})}$ for any odd prime number $p$ dividing $n$.  
\end{lem}
\begin{proof}
Let $p$ be a prime number dividing $n$ and $d$ the square-free part of $3^{2e} - 4q^n$.
Note that $d$ is a negative odd integer.
We can write $3^{2e} - 4q^n = a^2d$ for some positive odd integer $a$.\\
(i) \ We treat the case where $p \ge 3$.
It is sufficient to prove that $\alpha$ is not a $p$th power in $\mathcal{O}_{\mathbb{Q}(\sqrt{3^{2e} - 4q^n})}$.
Suppose that $\alpha$ is a $p$th power in $\mathcal{O}_{\mathbb{Q}(\sqrt{3^{2e} - 4q^n})}$.
We can write
$$\alpha = \biggl(\frac{u + v\sqrt{d}}{2}\biggr)^p$$
for some odd integers $u$ and $v$.
Then, 
$$\frac{3^e + a\sqrt{d}}{2} = \frac{1}{2^p}\biggl(\sum_{j = 0}^{\frac{p-1}{2}} \binom{p}{2j} u^{p-2j}v^{2j}d^j + w\sqrt{d}\biggr)$$
for some integer $w$.
Comparing the real parts of this equation,
$$2^{p-1}3^e = \sum_{j = 0}^{\frac{p-1}{2}} \binom{p}{2j} u^{p-2j}v^{2j}d^j = u\sum_{j = 0}^{\frac{p-1}{2}} \binom{p}{2j} u^{p-2j-1}v^{2j}d^j.$$
Then,
$$u \mid 2^{p-1}3^e.$$
Since $u$ is odd, $u = \pm 3^i$, where $i$ is an integer with $0 \le i \le e$.
Taking the norm of $\alpha$, we have
$$-v^2d + 3^{2i} = 4q^{n/p}.$$
We often use this.\\
(i-1) \ Assume $i = e$.
Then, 
$$-v^2d + 3^{2e} = 4q^{n/p}.$$
This implies that both $(x, y) = (a, n)$ and $(x, y) = (|v|, n/p)$ are positive integer solutions of the equation $-dx^2 + 3^{2e} = 4q^y$. 
By Lemma \ref{lem3.2}, this is a contradiction.\\
(i-2) \ Assume $i = 0$.
Then,
$$1 - 4q^{n/p} = v^2d.$$
Note that $3 \nmid v$.
In fact, we see from 
$$2^{p-1}3^e = u^p + u\sum_{j = 1}^{\frac{p-1}{2}} \binom{p}{2j} u^{p-2j-1}v^{2j}d^j$$
that $3$ divides $u$ if  $3$ divides $v$.
This is a contradiction.
Then,
$$v^2 \equiv 1 \bmod 6.$$
Since $3$ does not divide $a$, we also find $a^2 \equiv 1 \bmod 6$.
Then,
\begin{equation} \label{eq3.1}
1 - 4q^{n/p} = v^2d \equiv d \equiv a^2d = 3^{2e} -4q^n \equiv 3 - 4q^n \bmod 6.
\end{equation}
Since $p$ is odd, $n/p \equiv n \bmod 2$.
Then,
$$q^{n/p} \equiv q^n \bmod 6.$$
We see from this that equation (\ref{eq3.1}) is impossible.\\
(i-3) \ Assume $0 < i < e$.
Then, 
$$2^{p-1}3^e = \pm 3^i\sum_{j = 0}^{\frac{p-1}{2}} \binom{p}{2j} 3^{i(p-2j-1)}v^{2j}d^j,$$
that is,
$$\pm 2^{p-1}3^{e-i} = \sum_{j = 0}^{\frac{p-1}{2}} \binom{p}{2j} 3^{i(p-2j-1)}v^{2j}d^j.$$
Since
$$\sum_{j=0}^{\frac{p-1}{2}} \binom{p}{2j} 3^{i(p-2j-1)}v^{2j}d^j = pv^{p-1}d^{\frac{p-1}{2}} + \sum_{j=0}^{\frac{p-3}{2}} \binom{p}{2j} 3^{i(p-2j-1)}v^{2j}d^j$$
holds, $3$ divides $pv^{p-1}d^{\frac{p-1}{2}}$.
Note that $3 \nmid v$.
In fact, if $3$ divides $v$, then 
$$u^2 - v^2d = 4q^{n/p} \equiv 0 \bmod 3,$$
a contradiction.
Then, $p = 3$.
Therefore,
$$\pm 4 \cdot 3^{e-i} = \sum_{j = 0}^{1} \binom{3}{2j} 3^{i(2-2j)}v^{2j}d^j = 3^{2i} + 3v^2d,$$
that is, 
$$\pm 4\cdot 3^{e-i-1} = 3^{2i-1} + v^2d.$$
(i-3-1) \ We treat the case where $e \neq i + 1$.
In this case, $3$ divides $v^2d$, a contradiction.\\
(i-3-2) \ We treat the case where $e = i + 1$.
In this case, 
\begin{equation} \label{eq3.2}
\pm 4 = 3^{2i-1} + v^2d. 
\end{equation}
(i-3-2a) \ When $q$ is odd, 
$$d \equiv a^2d = 3^{2e} - 4q^n \equiv 5 \bmod 8.$$
Then, $3^{2i-1} + v^2d \equiv 0 \bmod 8$, a contradiction.\\
(i-3-2b) \ When $q = 2$, we see from equation (\ref{eq3.2}) and $3^{2i} - v^2d = 2^{(n/3)+2}$ that
$$3^{2i-1} = 2^{n/3} \pm1.$$
(i-3-2b-1) \ Suppose $n$ is even.
Since $3^{2i-1} \equiv 3 \bmod 4$ holds, it must be $3^{2i-1} = 2^{n/3} - 1$.
Then,
$$3^{2i-1} = (2^{n/6} + 1)(2^{n/6} - 1).$$
It follows from $\gcd(2^{n/6} + 1, 2^{n/6} - 1) = 1$ that $2^{n/6} + 1 = 3^{2i-1}$ and $2^{n/6} - 1 = 1$.
Then, $(n, i) = (6, 1)$ and $e = i + 1 = 2$.
The case where $(q, n, e) = (2, 6, 2)$ is removed.\\
(i-3-2b-2) \ Suppose $n$ is odd.
From $3^{2i-1} \equiv 3 \bmod 6$, it must be 
$$3^{2i-1} = 2^{n/3} + 1.$$
Since $n/3$ is odd, we can write 
$$2\{2^{(n-3)/6}\}^2 + 1 = 3^{2i-1}.$$
The positive integer solutions $(x, y)$ of the equation 
$$2x^2 + 1 = 3^y$$ 
are 
$$(x, y) = (1, 1), \ (2, 2), \ (11, 5)$$
(see \cite[Theorem 2.3]{LL}).
Then, 
$$(2^{(n-3)/6}, 2i - 1) = (1, 1), \ (2, 2), \ (11, 5).$$
We see 
$$(2^{(n-3)/6}, 2i - 1) \neq (2, 2), \ (11, 5)$$
easily.
If $(2^{(n-3)/6}, 2i - 1) = (1, 1)$, then $(n, i) = (3, 1)$, that is, $(n, e) = (3, 2)$.
When $(n, e) = (3, 2)$, we have $3^{2e} -4q^n = 49 > 0$, a contradiction.\\
\hspace{17pt}Therefore, $\pm \alpha$ is not a $p$th power in $\mathcal{O}_{\mathbb{Q}(\sqrt{3^{2e} - 4q^n})}$ for any odd prime number $p$ dividing $n$
when $(q, n, e) \neq (2, 6, 2)$.\\
(ii) \ We treat the case where $p = 2$.
Suppose that $\pm \alpha$ is a square number in $\mathcal{O}_{\mathbb{Q}(\sqrt{3^{2e} - 4q^n})}$.
We can write
$$\pm \alpha = \biggl(\frac{u + v\sqrt{d}}{2}\biggr)^2$$
for some odd integers $u$ and $v$.
It follows from 
$$\biggl(\frac{u + v\sqrt{d}}{2}\biggr)^2 = \frac{u^2 + v^2d}{4} + \frac{uv}{2}\sqrt{d}$$ 
that
\begin{equation} \label{eq3.3}
u^2 + v^2d = \pm 2 \cdot 3^e
\end{equation}
and
$$uv = \pm a.$$
Since $3$ does not divide $a$, we see that $3$ does not divide $uv$.
Then, 
$$u^2 \equiv v^2 \equiv 1 \bmod 6.$$
Taking the norm of $\pm \alpha$, we have
\begin{equation} \label{eq3.4}
4q^{n/2} = u^2 - v^2d.
\end{equation}
(ii-1) \ We treat the case where $q \equiv 1 \bmod 3$ or $n \equiv 0 \bmod 4$.
In this case, $4q^{n/2} \equiv 4 \bmod 6$.
Then, 
$$4 \equiv 4q^{n/2} = u^2 - v^2d \equiv 1 - d \bmod 6,$$
that is, $d \equiv 3 \bmod 6$, a contradiction.\\
(ii-2) \ We treat the case where $q \equiv 2 \bmod 3$ with $q \neq 2$ and $n \equiv 2 \bmod 4$.
It follows from equations (\ref{eq3.3}) and (\ref{eq3.4}) that
$$u^2 = 2q^{n/2} \pm 3^e.$$
Since $u^2 \equiv 1 \bmod 4$ and $2q^{n/2} \equiv 2 \bmod 4$ holds,
$u^2 = 2q^{n/2} + 3^e$ if $e$ is odd and $u^2 = 2q^{n/2} - 3^e$ otherwise.
Then,
$$u^2 = 2q^{n/2} - (-3)^e.$$
If $2q^{n/2} - (-3)^e$ is not a square number, this is impossible.\\
(ii-3) \ We treat the case where $q = 2$ and $n \equiv 2 \bmod 4$.
It follows from equations (\ref{eq3.3}) and (\ref{eq3.4}) that
$$u^2 = 2^{(n/2)+1} \pm 3^e.$$
Since $u^2 \equiv 1 \bmod 4$ and $2^{(n/2)+1} \equiv 0 \bmod 4$ holds,
$u^2 = 2^{(n/2)+1} - 3^e$ if $e$ is odd and $u^2 = 2^{(n/2)+1} + 3^e$ otherwise.
Then, if $e$ is odd and $2^{(n/2)+1} - 3^e$ is not a square number, this is impossible.
We will show $u^2 = 2^{(n/2)+1} + 3^e$ is impossible if $e$ is even and $(n, e) \neq (6, 2)$ as follows.
When $e$ is even,
$$(u + 3^{e/2})(u - 3^{e/2}) = 2^{(n/2)+1}.$$
We see from $\gcd(u + 3^{e/2}, u - 3^{e/2}) = 2$ that 
$$(u + 3^{e/2}, u - 3^{e/2}) = (2^{n/2}, 2) \ \ \ {\rm or} \ \ \ (-2, -2^{n/2}).$$
Then,
$$3^{e/2} + 1 = 2^{(n-2)/2}.$$
When $n = 2$, we have $3^{e/2} + 1 = 1$.
This is impossible.
Therefore, we treat the case $n \ge 6$.\\
(ii-3-1) \ Assume $e \equiv 0 \bmod 4$.
We can write 
$$(3^{e/4})^2 + 1 = 2^{(n-2)/2}.$$
We see from $n \ge 6$ that $\dfrac{n - 2}{2} \ge 2$.
Then,
$$(3^{e/4})^2 + 1 \equiv 1 + 1 \equiv 2 \bmod 4$$
and 
$$2^{(n-2)/2} \equiv 0 \bmod 4.$$
This is a contradiction.\\
(ii-3-2) \ Assume $e \equiv 2 \bmod 4$.
We can write
$$3\cdot \{3^{(e-2)/4}\}^2 + 1 = 2^{(n-2)/2}.$$
When $n = 6$, we have $3\cdot \{3^{(e-2)/4}\}^2 + 1 = 4$.
Then, $e = 2$.
The case where $(q, n, e) = (2, 6, 2)$ is removed.
When $n \ge 10$, we have $\dfrac{n - 2}{2} \ge 4$.
Then,
$$3\cdot \{3^{(e-2)/4}\}^2 + 1 \equiv 3 + 1 \equiv 4 \bmod 8$$
and
$$2^{(n-2)/2} \equiv 0 \bmod 8.$$ 
This is a contradiction.\\
\hspace{17pt}Then, $u^2 = 2^{(n/2)+1} + 3^e$ is impossible when $e$ is even and $(n, e) \neq (6, 2)$.         
Therefore, $\pm \alpha$ is not a square number in $\mathcal{O}_{\mathbb{Q}(\sqrt{3^{2e} - 4q^n})}$ if $q$, $n$, and $e$ satisfy 
one of the following conditions:\\ 
$(1)$ \ Case $q > 3$.\\
\hspace{10pt}$(1.a') \ n \equiv 0 \bmod 4$.\\
\hspace{10pt}$(1.b) \ q \equiv 1 \bmod 3$.\\
\hspace{10pt}$(1.c) \ q \equiv 2 \bmod 3$, $n \equiv 2 \bmod 4$, and $2q^{n/2} - (-3)^e$ is not a square number.\\
$(2)$ \ Case $q = 2$.\\
\hspace{10pt}$(2.a') \ n \equiv 0 \bmod 4$.\\
\hspace{10pt}$(2.b) \ n \equiv 2 \bmod 4$ and $e \equiv 0 \bmod 2$ with $(n, e) \neq (6, 2)$.\\
\hspace{10pt}$(2.c) \ n \equiv 2 \bmod 4$, $e \equiv 1 \bmod 2$, and $2^{(n/2)+1} - 3^e$ is not a square number.\\
The proof of Lemma \ref{lem3.3} is completed.
\end{proof}
\subsection{Proof of Theorem \ref{thm6}} \label{sec3.3}
In this section, we show Theorem \ref{thm6} by using Lemma \ref{lem3.3}.\\
~\\
{\bf Proof of Theorem \ref{thm6}.} \ \ Since $\gcd (q, 3^{2e} - 4q^n) = 1$ and $q \nmid \alpha$ hold,
the prime number $q$ splits in $\mathbb{Q}(\sqrt{3^{2e} - 4q^n})$.
Note that ideals $(\alpha)$ and $(\overline{\alpha})$ are coprime and $N(\alpha) = q^n$.
Then,
$$(\alpha) = \wp^n,$$
where $\wp$ is the prime ideal of $\mathcal{O}_{\mathbb{Q}(\sqrt{3^{2e} - 4q^n})}$ over $q$.
Let $s$ be the order of the ideal class containing $\wp$.
We can write $n = sn'$ for some integer $n'$.
Since $\wp^s$ is principal, there exists some element $\beta \in \mathcal{O}_{\mathbb{Q}(\sqrt{3^{2e} - 4q^n})}$
such that $\wp^s = (\beta)$.
Then,
$$(\alpha) = (\wp^s)^{n'} = (\beta)^{n'} = (\beta^{n'}).$$
We see $\mathbb{Q}(\sqrt{3^{2e} - 4q^n}) \neq \mathbb{Q}(\sqrt{-3})$ easily.
When $q$ is odd, we have
$$d \equiv a^2d = 3^{2e} - 4q^n \equiv -3 \bmod 8.$$
When $q = 2$, we have
$$d \equiv a^2d = 3^{2e} - 2^{n+2} \equiv -7 \bmod 8.$$
Then,
$$\mathbb{Q}(\sqrt{3^{2e} - 4q^n}) \neq \mathbb{Q}(\sqrt{-1}), \mathbb{Q}(\sqrt{-3}),$$
that is,
$$\mathcal{O}_{\mathbb{Q}(\sqrt{3^{2e} - 4q^n})}^{\times} = \{\pm 1\}.$$
This implies that
$$\pm \alpha = \beta^{n'}.$$

First, we treat the case where $q \equiv 1 \bmod 3$ or $n \not \equiv 2 \bmod 4$.
It follows from Lemma \ref{lem3.3} (1.a), (1.b), and (2.a) that $n' = 1$.
Then,
$$s = n.$$

Secondly, we treat the case where $q \equiv 2 \bmod 3$ with $q \neq 2$ and $n \equiv 2 \bmod 4$.
It follows from Lemma \ref{lem3.3} (1.c) that $n' = 1$ if $2q^{n/2} - (-3)^e$ is not a square number.
Then, 
$$s = n.$$
On the other hand, it follows from Lemma \ref{lem3.3} (1.d) that $n'$ is $1$ or a power of two if $2q^{n/2} - (-3)^e$ is a square number.
By the assumption $n \equiv 2 \bmod 4$, we have $n' = 1$ or $2$.
Then, 
$$s = n \ \ \ {\rm or} \ \ \ n / 2.$$

Finally, we treat the case where $q = 2$ and $n \equiv 2 \bmod 4$ with $(n, e) \neq (6, 2)$.
Assume $e \equiv 0 \bmod 2$.
It follows from Lemma \ref{lem3.3} (2.b) that $n' = 1$.
Then, 
$$s = n.$$
Next assume $e \equiv 1 \bmod 2$.
It follows from Lemma \ref{lem3.3} (2.c) that $n' = 1$ if $2^{(n/2)+1} - 3^e$ is not a square number.
Then, 
$$s = n.$$
On the other hand, it follows from Lemma \ref{lem3.3} (2.d) that $n'$ is $1$ or a power of two if $2^{(n/2)+1} - 3^e$ is a square number.
By the assumption $n \equiv 2 \bmod 4$, we have $n' = 1$ or $2$.
Then, 
$$s = n \ \ \ {\rm or} \ \ \ n / 2.$$
The proof of Theorem \ref{thm6} is completed.
\section{Other Results} \label{sec4}
In this section , we show other two results on the divisibility of the class numbers of $\mathbb{Q}(\sqrt{x^2 - 4k^n})$.
First, we show the following theorem.
\begin{thm} \label{thm4.1}
Let $n$ be an integer greater than $2$, $k$ an integer greater than $1$, and
$x$ a positive odd integer such that $\gcd(k, x) = 1$ and $x^2 < 4k^n$.
Assume $x^2 - 4k^n = a^2d$, where $a$ is a positive integer and
$d$ is a square-free integer less than $-3$.
If $k^n <\dfrac{(1-d)^2}{16}$, then $n \mid h(d)$.
\end{thm}
\begin{proof}
The method of the proof is based on the one in \cite[Proposition 1]{So}.
We can write
$$(k)^n = \biggl(\frac{x + a\sqrt{d}}{2}\biggr)\biggl(\frac{x - a\sqrt{d}}{2}\biggr)$$
in $\mathbb{Q}(\sqrt{d})$.
Put $\eta := \dfrac{x + a\sqrt{d}}{2}$.
Then, $\eta \in \mathcal{O}_{\mathbb{Q}(\sqrt{d})}$.
Let $r$ be any prime number dividing $k$.
Since $\gcd(r, d) = 1$ and $r \nmid \eta$ hold,
the prime number $r$ splits in $\mathbb{Q}(\sqrt{d})$.
Note that ideals $(\eta)$ and $(\overline{\eta})$ are coprime and $N(\eta) = k^n$,
where $\overline{\eta}$ is the conjugate of $\eta$.
Then,
$$(\eta) = \mathfrak{I}^n$$
for some ideal $\mathfrak{I}$ of $\mathcal{O}_{\mathbb{Q}(\sqrt{d})}$.
Let $s$ be the order of the ideal class containing $\mathfrak{I}$.
Suppose $s \neq n$.
Since $\mathfrak{I}^s$ is principal, we can write
$$\mathfrak{I}^s = \biggl(\frac{u + v\sqrt{d}}{2}\biggr)$$
for some odd integers $u$ and $v$.
We see from $d < -3$ that $\mathcal{O}_{\mathbb{Q}(\sqrt{d})}^{\times} = \{\pm 1\}$.
Then,
$$\eta = \pm \biggl(\frac{u + v\sqrt{d}}{2}\biggr)^{n/s}.$$
Since $x$ and $a$ are non-zero integers, we have $|u| \ge 1$ and $|v| \ge 1$.
Then,
$$N(\eta) = \biggl(\frac{u^2 - v^2d}{4}\biggr)^{n/s} \ge \frac{(1 - d)^2}{16},$$
a contradiction.
Therefore, $s = n$.
The proof of Theorem \ref{thm4.1} is completed.
\end{proof}
Next, we show the following theorem which is regarded as a generalization of a result of J.~H.~E.~Cohn~\cite{Co2}.
\begin{thm} \label{thm4.2}
Let $n$ be a positive integer, $l$ an odd prime number, and $e$ a non-negative integer.
If $(n, e) \neq (4, 0)$, then the class numbers of the imaginary quadratic fields
 $\mathbb{Q}(\sqrt{1 - 4(2l^e)^n})$ are divisible by $n$.
\end{thm}
\begin{proof}
The method of the proof is based on the one in \cite[Theorem]{Co2}.
The case where $e = 0$ is treated in \cite[Theorem]{Co2}.
Here, we assume $e > 0$.
Let $d$ be the square-free part of $1 - 4(2l^e)^n$.
Note that $d$ is a negative odd integer.
We can write $1 - 4(2l^e)^n = a^2d$ for some positive odd integer $a$.
We see
$$(2l^e)^n = \biggl(\frac{1 + a\sqrt{d}}{2}\biggr)\biggl(\frac{1 - a\sqrt{d}}{2}\biggr)$$
in $\mathbb{Q}(\sqrt{d})$.
Put $\xi := \dfrac{1 + a\sqrt{d}}{2}$.
Then, $\xi \in \mathcal{O}_{\mathbb{Q}(\sqrt{d})}$.
Since $\gcd(2l, d) = 1$, $2 \nmid \xi$, and $l \nmid \xi$ hold,
the prime numbers $2$ and $l$ split in $\mathbb{Q}(\sqrt{d})$.
Note that ideals $(\xi)$ and $(\overline{\xi})$ are coprime and $N(\xi) = (2l^e)^n$,
where $\overline{\xi}$ is the conjugate of $\xi$.
Then,
$$(\xi) = \mathfrak{J}^n$$
for some ideal $\mathfrak{J}$ of $\mathcal{O}_{\mathbb{Q}(\sqrt{d})}$.
Let $p$ be any prime number dividing $n$.
When $p$ is odd, we can show that $\pm \xi$ is not a $p$th power in $\mathcal{O}_{\mathbb{Q}(\sqrt{d})}$
in a way similar to the proof in \cite[Theorem]{Co2}.
Now, we treat the case where $p = 2$.
Assume
\begin{equation} \label{eq4.1}
\pm \xi = \biggl(\frac{u + v\sqrt{d}}{2}\biggr)^2
\end{equation}
for some odd integers $u$ and $v$.
Expanding the right side of equation (\ref{eq4.1}), 
$$\pm 2 = u^2 + v^2d$$
and 
$$\pm a = uv.$$
It follows from $d \equiv 1 \bmod 8$ that 
$$2 = u^2 + v^2d$$
and
$$a = uv.$$
Then,
$$4(2l^e)^n = 1 - du^2v^2 = 1 + u^2(u^2 - 2) = (u^2 - 1)^2.$$
Therefore,
$$2(2l^e)^{n/2} = (u + 1)(u - 1).$$
It follows from $\gcd(u + 1, u - 1) = 2$ and $u + 1 > u - 1$ that the following cases are possible:
(i) \ $u + 1 = (2l^e)^{n/2}$ \ {\rm and} \ $u - 1 = 2$, (ii) \ $u + 1 = 2(l^e)^{n/2}$ \ {\rm and} \ $u - 1 = 2^{n/2}$,
(iii) \ $u + 1 = -2$ \ {\rm and} \ $u - 1 = -(2l^e)^{n/2}$, (iv) \ $u + 1 = -2^{n/2}$ \ {\rm and} \ $u - 1 = -2(l^e)^{n/2}$.

First, we treat the case (i).
In this case, $u = 3$. 
Since  $e$ is positive, this is impossible.
Similarly, we see that the case (iii) is impossible.
Next, we treat the case (ii).
In this case, we have
$$1 = (l^e)^{n/2} - 2^{(n-2)/2}.$$
Since
$$(l^e)^{n/2} \ge l^{n/2} > 2^{n/2} > 2^{(n-2)/2}$$
holds, this is impossible.
Similarly, we see that the case (iv) is impossible.
Therefore, $\pm \xi$ is not a square number in $\mathcal{O}_{\mathbb{Q}(\sqrt{d})}$ when $e > 0$.
Let $s$ be the order of the ideal class containing $\mathfrak{J}$.
We can write $n = sn'$ for some integer $n'$.
Since $\mathfrak{J}^s$ is principal, there exists some element $\iota \in \mathcal{O}_{\mathbb{Q}(\sqrt{d})}$
such that $\mathfrak{J}^s = (\iota)$. Then,
$$(\xi) = (\mathfrak{J}^s)^{n'} = (\iota)^{n'} = (\iota^{n'}).$$
We see from $d \equiv 1 \bmod 8$ that
$$\mathbb{Q}(\sqrt{1 - 4(2l^e)^n}) \neq \mathbb{Q}(\sqrt{-1}), \mathbb{Q}(\sqrt{-3}).$$
Then,
$$\pm \xi = \iota^{n'}.$$
By the above discussion, we have $n' = 1$.
Then, $s = n$.
The proof of Theorem \ref{thm4.2} is completed.
\end{proof}
\subsection*{Addendum}
We give an addendum about Theorem \ref{thm5}.
After submitting this paper, a result of K.~Ishii~\cite{Is} was published.
Let $k$ be an integer greater than $1$.
In his paper, he proved that if $n$ is an even integer greater than $5$, then the class number of $\mathbb{Q}(\sqrt{1 - 4k^n})$
is divisible by $n$ except $(k, n) = (13, 8)$.
His theorem does not cover the cases where $n = 2$ or $n = 4$.
We treat these cases for odd integers $k$ in Theorem \ref{thm5}.
\subsection*{Acknowledgements}
The author wishes to express sincere gratitude to Professor Manabu Ozaki, Doctor Satoshi Fujii, and Doctor Takayuki Morisawa for helpful discussions.
She would like to thank Professor Ryotaro Okazaki for answering her questions.
She would also like to thank Doctor Yanyan Wang and Hiroyuki Mitani for helping me to read Chinese papers.
Further, she would like to express my thanks to Professor Yasuhiro Kishi and Professor Tsuyoshi Itoh
for careful reading this paper and several useful comments.
Finally, she wishes to be grateful to Professor Kohji Matsumoto and Professor Yasushi Mizusawa for continuous encouragement.

~\\
Graduate School of Mathematics\\
Nagoya University\\
Chikusa-ku, Nagoya City\\
Aichi 464-8602, Japan\\
E-mail: m07004a@math.nagoya-u.ac.jp
\end{document}